\documentclass[11pt]{amsart}
\usepackage{amsfonts}
\usepackage{amsmath, amsthm}

\usepackage{comment}
\usepackage{tensor}

\numberwithin{equation}{section}

\newtheorem{Theorem}{Theorem}[section]
\newtheorem{Lemma}{Lemma}[section]
\newtheorem{Definition}[Lemma]{Definition}

\newcommand{\C}{\mathbb{C}}

\renewcommand{\O}{\mathcal{O}}

\newcommand{\R}{\mathbb{R}}

\renewcommand{\S}{\mathbb{S}}

\newcommand{\del}{\partial}

\renewcommand{\phi}{\varphi}
\newcommand{\grad}{\nabla}
\renewcommand{\epsilon}{\varepsilon}

\newcommand{\Lip}{\mathrm{Lip}}

\title[An annulus and a half-helicoid maximize eigenvalues]{An annulus and a half-helicoid maximize Laplace eigenvalues}
\author{Sinan Ariturk}
\date{}

\begin{document}

\begin{abstract}
The Dirichlet eigenvalues of the Laplace-Beltrami operator are larger on an annulus than on any other surface of revolution in $\R^3$ with the same boundary.
This is established by defining a sequence of shrinking cylinders about the axis of symmetry and proving that flattening a surface outside of each cylinder succesively increases the eigenvalues.
A similar argument shows that the Dirichlet eigenvalues of the Laplace-Beltrami operator are larger on a half-helicoid than on any other screw surface in $\R^2 \times \S^1$ with the same boundary.
\end{abstract}

\maketitle

\section{Introduction}
Let $\Sigma$ be a compact connected smoothly immersed surface of revolution in $\R^3$ with two boundary components.
Let $\Delta_\Sigma$ be the Laplace-Beltrami operator on $\Sigma$.
Denote the Dirichlet eigenvalues of $-\Delta_\Sigma$ by
\[
	0 < \lambda_1(\Sigma) < \lambda_2(\Sigma) \le \lambda_3(\Sigma) \le \ldots
\]
Let $R_1$ and $R_2$ be the radii of the boundary components of $\Sigma$.
Assume
\[
	R_1 > R_2
\]
Let $A$ be an annulus in $\R^2$ with outer radius $R_1$ and inner radius $R_2$.
Denote the Dirichlet eigenvalues of the Laplacian on $A$ by
\[
	0 < \lambda_1(A) < \lambda_2(A) \le \lambda_3(A) \le \ldots
\]

\begin{Theorem}
\label{annulus}
If $\Sigma$ is not isometric to $A$, then, for every $j=1,2,3,\ldots,$
\[
	\lambda_j(\Sigma) < \lambda_j(A)
\]
\end{Theorem}

This inequality does not extend to surfaces which are not surfaces of revolution.
That is, there are compact connected smoothly embedded surfaces in $\R^3$, with two circular boundary components of radii $R_1$ and $R_2$,  which are not surfaces of revolution and have larger Dirichlet eigenvalues than the annulus.
These surfaces can be realized as graphs over an annulus.
This can be proven with Berger's variational formula \cite{Be}.

Theorem \ref{annulus} extends a similar result for a disc \cite{Adisc}.
That result shows that the Dirichlet eigenvalues of a compact connected surface of revolution in $\R^3$ with one boundary component are smaller than the Dirichlet eigenvalues of a planar disc with the same boundary.

To prove Theorem \ref{annulus}, we define a sequence of shrinking cylinders about the axis of symmetry.
Then we describe a procedure to flatten a surface outside of each cylinder.
We prove that repeating this for each cylinder successively increases the eigenvalues, and the theorem follows after finitely many iterations.

The argument can also be applied to establish a similar result for a half-helicoid in $\R^2 \times \S^1$.
To state this, consider the screw action of $\S^1$ on $\R^2 \times \S^1$.
Identifying $\R^2 \times \S^1$ with a subspace of $\C^2$, the screw action is defined by scalar multiplication.
Fix constants
\[
	R_1' > R_2' \ge 0
\]
Let $H$ be the helicoidal strip
\[
	H = \bigg\{ (t \cos \theta, t \sin \theta, e^{i\theta}) : R_2' \le t \le R_1', 0 \le \theta \le 2 \pi \bigg\}
\]
Let $\Delta_H$ be the Laplace-Beltrami operator on $H$, and denote the Dirichlet eigenvalues of $-\Delta_H$ by
\[
	0 < \lambda_1(H) < \lambda_2(H) \le \lambda_3(H) \le \ldots
\]
Let $S$ be a compact connected smoothly immersed surface in $\R^2 \times \S^1$ which is invariant under the screw action of $\S^1$.
Assume that the boundary of $S$ is non-empty and isometric to the boundary of $H$.
Let $\Delta_S$ be the Laplace-Beltrami operator on $S$, and denote the Dirichlet eigenvalues of $-\Delta_S$ by
\[
	0 < \lambda_1(S) < \lambda_2(S) \le \lambda_3(S) \le \ldots
\]

\begin{Theorem}
\label{helicoid}
If $S$ is not isometric to $H$, then for $j=1,2,3,\ldots$,
\[
	\lambda_j(S) < \lambda_j(H)
\]
\end{Theorem}

In the next section, we reformulate both Theorem \ref{annulus} and Theorem \ref{helicoid} as statements about curves in a half-plane.
Then we consider a more general statement which implies both results.
In the third section, we begin the proof and reduce the problem to a special case.
In the fourth section, we conclude the proof by establishing this special case.

These problems are related to many other results on  optimization of Laplace eigenvalues.
The most classical of these problems is to minimize the Dirichlet eigenvalues of the Laplacian among open domains of fixed volume in Euclidean space.
The Faber-Krahn theorem states that a ball minimizes the first Dirichlet eigenvalue among such domains.
The Krahn-Szeg\"o theorem states that the union of two disjoint balls with equal radii minimizes the second Dirichlet eigenvalue.
Bucur and Henrot proved that there exists a quasi-open set which minimizes the third eigenvalue \cite{BH}.
This was extended to higher eigenvalues by Bucur \cite{Bu}.

Another problem is to fix a Euclidean ball and remove a smaller ball.
The radii of the balls are fixed and only the position of the smaller ball varies.
In two dimensions, Hersch proved that the concentric annulus maximizes the first Dirichlet eigenvalue among these domains \cite{H1}.
This was extended to higher dimensions by Harrell, Kr\"oger, and Kunata, as well as by Kesavan \cite{HKK, K}.
El Soufi and Kiwan proved that the concentric annulus maximizes the second Dirichlet eigenvalue among these domains \cite{EK}.

A related problem for the Laplace-Beltrami operator is to fix a two dimensional surface and maximize the Laplace eigenvalues among Riemannian metrics of fixed area.
Hersch proved that the canonical metric on $\S^2$ maximizes the first non-zero eigenvalue \cite{H2}.
On an orientable surface, Yang and Yau obtained upper bounds on the first eigenvalue, depending on the genus \cite{YY}.
These were extended to non-orientable surfaces by Li and Yau.
Li and Yau also showed that the canonical metric on $\mathbb{RP}^2$ maximizes the first non-zero eigenvalue \cite{LY}.
Nadirashvili proved the same is true for the flat equilateral torus, whose fundamental parallelogram is comprised of two equilateral triangles \cite{N1}.
It is not known if there is such a maximal metric on the Klein bottle, but Jakobson, Nadirashvili, and Polterovich showed there is a critical metric \cite{JNP}.
El Soufi, Giacomini, and Jazar proved this is the only critical metric on the Klein bottle \cite{EGJ}.
Nadirashvili proved that the second non-zero eigenvalue for any metric on $\S^2$ is less than the first non-zero eigenvalue of the canonical metric with half the area \cite{N2}.
Nadirashvili and Sire studied maximizing eigenvalues among metrics of the same area in a conformal class \cite{NS1, NS2}.
Petrides considered maximization in a conformal class and also proved an existence result for an eigenvalue-maximizing metric on an orientable surface \cite{P}.

For a closed compact hypersurface in $\R^{n+1}$, Chavel and Reilly obtained upper bounds for the first non-zero eigenvalue in terms of the surface area and the volume of the enclosed domain \cite{Ch,R}.
Abreu and Freitas proved that for a metric on $\S^2$ which can be isometrically embedded in $\R^3$ as a surface of revolution, the $\S^1$-invariant eigenvalues  are bounded by $\S^1$-invariant eigenvalues on a flat disc with half the area \cite{AF}.
Colbois, Dryden, and El Soufi extended this to metrics on $\S^n$ which can be isometrically embedded in $\R^{n+1}$ as hypersurfaces of revolution \cite{CDE}.

\section{Reformulation}

In this section, we reformulate both Theorem \ref{annulus} and Theorem \ref{helicoid} as statements about curves in a half-plane.
We reformulate Theorem \ref{annulus} in the first subsection and Theorem \ref{helicoid} in the second subsection.
In the third subsection, we consider a more general statement which implies both Theorem \ref{annulus} and Theorem \ref{helicoid}.
Then the rest of the paper is dedicated to proving this general statement.

\subsection{Reformulation of Theorem 1.1}

Fix a plane in $\R^3$ containing the axis of symmetry of $\Sigma$.
Identify this plane with $\R^2$, so that the axis of symmetry is identified with the axis
\[
	\bigg\{ (0,y) \in \R^2 : y \in \R \bigg\}
\]
Let $\R^2_+$ denote the half-plane given by
\[
	\R^2_+ = \bigg\{ (x,y) \in \R^2 : x > 0 \bigg\}
\]
We may assume that the boundary component of $\Sigma$ which has radius $R_1$ intersects $\R^2$ at the point $(R_1,0)$.
Let $L_\Sigma$ be the length of a meridian of $\Sigma$ and let
\[
	\alpha_\Sigma:[0,L_\Sigma] \to \R^2_+
\]
be a regular arc length parametrization of the meridian $\Sigma \cap \R^2_+$ with
\[
	\alpha_\Sigma(0)=(R_1,0)
\]

For each $j=1,2,3,\ldots$, the eigenvalues $\lambda_j(\Sigma)$ can be characterized as
\[
	\lambda_j(\Sigma) = \min_V \max_{f \in V} \bigg\{ \frac{ \int_\Sigma | \grad f |^2}{\int_\Sigma |f|^2} \bigg\}
\]
Here the minimum is taken over all $j$-dimensional subspaces $V$ of $C_0^\infty(\Sigma)$, the space of smooth real-valued functions on $\Sigma$ which vanish on $\del \Sigma$.
The minimum is attained by the subspace generated by the first $j$ eigenfunctions.
Note that since $\Sigma$ is a surface of revolution, separation of variables reduces the eigenfunction equation to an ordinary differential equation.

Write $\alpha_\Sigma=(F_\Sigma,G_\Sigma)$, i.e. let $F_\Sigma$ and $G_\Sigma$ be the component functions of $\alpha_\Sigma$.
For a non-negative integer $k$ and a positive integer $n$, define
\[
	\lambda_{k,n}(\alpha_\Sigma) = \min_W \max_{w \in W} \left\{ \frac{ \int_0^{L_\Sigma} |w'|^2 F_\Sigma + \frac{k^2 |w|^2}{F_\Sigma} \,dt}{ \int_0^{L_\Sigma} |w|^2 F_\Sigma \,dt} \right\}
\]
Here the minimum is taken over all $n$-dimensional subspaces $W$ of $C_0^\infty(0,L_\Sigma)$, the space of smooth real-valued functions on $[0,L_\Sigma]$ which vanish at the endpoints.
It follows from separation of variables that
\[
	\bigg\{ \lambda_{k,n}(\alpha_\Sigma) \bigg\} = \bigg\{ \lambda_j(\Sigma) \bigg\}
\]
Moreover, if the eigenvalues $\lambda_{k,n}(\alpha_\Sigma)$ are counted twice for $k\neq 0$, then the multiplicities agree.

Define a curve
\[
	\omega_A:[0, R_1-R_2] \to \R^2_+
\]
by
\[
	\omega_A(t) = (R_1-t,0)
\]
Then $\omega_A$ parametrizes a meridian of an annulus, isometric to $A$.
Define $\lambda_{k,n}(\omega_A)$ similarly to $\lambda_{k,n}(\alpha_\Sigma)$.
Then
\[
	\bigg\{ \lambda_{k,n}(\omega_A) \bigg\} = \bigg\{ \lambda_j(A) \bigg\}
\]
Again, if the eigenvalues $\lambda_{k,n}(\omega_A)$ are counted twice for $k\neq 0$, then the multiplicities agree.
Now, to prove Theorem \ref{annulus} it suffices to prove the following lemma.

\begin{Lemma}
\label{knannulus}
Assume that $\alpha_\Sigma$ is not equal to $\omega_A$.
Then for all non-negative $k$ and all positive $n$,
\[
	\lambda_{k,n}(\alpha_\Sigma) < \lambda_{k,n}(\omega_A)
\]
\end{Lemma}

\subsection{Reformulation of Theorem 1.2}
Let $\O$ denote the space of orbits of the screw action on $\R^2 \times \S^1$.
There is a differentiable structure and a Riemannian metric on $\O$ such that the projection map $\R^2 \times \S^1 \mapsto \O$ is a Riemannian submersion.
Let
\[
	\R^2_+ = \bigg\{ (r,\theta) \in \R^2 : r >0 \bigg\}
\]
Define a map from $\overline{\R^2_+}$ to $\O$ by
\[
	(r,\theta) \mapsto [r \cos \theta, r \sin \theta, 1 ]
\]
Here $[r \cos \theta, r \sin \theta,1]$ denotes the orbit containing $(r \cos \theta,r \sin \theta,1)$.
Let $g^*$ be the pullback metric on $\R^2_+$, which takes the form
\[
		g^*= dr^2 + \bigg(\frac{r^2}{1+r^2} \bigg) d\theta^2
\]

The surface $H$ in $\R^2 \times \S^1$ projects to a geodesic in $\O$.
This geodesic can be lifted to a curve in $\R^2_+$, parametrized by a function
\[
	\omega_H:[0,R_1'-R_2'] \to \overline{\R^2_+}
\]
defined by
\[
	\omega_H(t) = (R_1'-t,0)
\]

The surface $S$ similarly projects to a curve in $\O$.
Recall the Dirichlet eigenvalues satisfy the following domain monotonicity property.
If $\Omega$ is a compact Riemannian manifold with smooth boundary and $U$ is a smoothly bounded proper open subset of $\Omega$, then the Dirichlet eigenvalues of $U$ are larger than the Dirichlet eigenvalues of $\Omega$.
Therefore, to prove Theorem \ref{helicoid}, it suffices to consider the case where $[0,0,1]$ is not an interior point of the curve in $\O$ associated to $S$.
Lift this curve to $\overline{\R^2_+}$.
Let $L_S$ be the length of this curve with respect to $g^*$, and let
\[
	\alpha_S:[0,L_S] \to \overline{\R^2_+}
\]
be an arc length parametrization.
We may assume that
\[
	\alpha_S(0)=(R_1',0)
\]
Also lift the curve in $\O$ associated to $S$ to a curve in $\R^2 \times \S^1$, parametrized by a function
\[
	\tilde \alpha_S:[0,L_S] \to \R^2 \times \S^1
\]
which intersects orbits of $\S^1$ orthogonally.
There is a diffeomorphism from $[0,L_S] \times \S^1$ to $S$ defined by
\[
	(t,e^{is}) \mapsto e^{is} \cdot \tilde \alpha_S(t)
\]
Let $h^*$ be the pullback metric on $[0,L_S] \times \S^1$.
Write $\alpha_S=(F_S,G_S)$, i.e. let $F_S$ and $G_S$ be the component functions of $\alpha_S$.
The metric $h^*$ takes the form
\[
		h^*= dt^2 + ( | F_S |^2 + 1) ds^2
\]

For each $j=1,2,3,\ldots,$ the eigenvalues $\lambda_j(S)$ can be characterized as
\[
	\lambda_j(S) = \min_V \max_{f \in V} \bigg\{ \frac{ \int_S | \grad f |^2 }{\int_S |f|^2 } \bigg\}
\]
Here the minimum is taken over all $j$-dimensional subspaces $V$ of $C_0^\infty(S)$, the space of smooth real-valued functions on $S$ which vanish on $\del S$.
The minimum is attained by the subspace generated by the first $j$ eigenfunctions.
Note that separation of variables reduces the eigenfunction equation to an ordinary differential equation.
For a non-negative integer $k$ and a positive integer $n$, define
\[
	\lambda_{k,n}(\alpha_S) = \min_W \max_{w \in W} \left\{ \frac{\int_0^{L_S} | w' |^2 (1+|F_S|^2)^{1/2} + \frac{k^2 w^2}{(1+|F_S|^2)^{1/2}} \,dt}{\int_0^{L_S} w^2 (1+|F_S|^2)^{1/2} \,dt} \right\}
\]
Here the minimum is taken over all $n$-dimensional subspaces $W$ of $C_0^\infty(0,L_S)$.
It follows that
\[
	\bigg\{ \lambda_{k,n}(\alpha_S) \bigg\} = \bigg\{ \lambda_j(S) \bigg\}
\]
Moreover, if the eigenvalues $\lambda_{k,n}(\alpha_S)$ are counted twice for $k\neq 0$, then the multiplicities agree.

Define $\lambda_{k,n}(\omega_H)$ similarly to $\lambda_{k,n}(\alpha_S)$.
Then
\[
	\bigg\{ \lambda_{k,n}(\omega_H) \bigg\} = \bigg\{ \lambda_j(H) \bigg\}
\]
Again, if the eigenvalues $\lambda_{k,n}(\omega_H)$ are counted twice for $k \neq 0$, then the multiplicities agree.
Now to prove Theorem \ref{helicoid}, it suffices to prove the following lemma.

\begin{Lemma}
\label{knheli}
If $\alpha_S$ is not equal to $\omega_H$, then for any non-negative integer $k$ and any positive integer $n$,
\[
	\lambda_{k,n}(\alpha_S) < \lambda_{k,n}(\omega_H)
\]
\end{Lemma}

\subsection{General Statement}
In this subsection, we consider a more general problem that implies Lemma \ref{knannulus} and Lemma \ref{knheli}.
Let
\[
	\R^2_+ = \bigg\{ (x,y) \in \R^2 : x>0 \bigg\}
\]
Let $g$ be a Riemannian metric on $\R^2_+$ of the form
\[
	g = dx^2 + g_{yy}(x) \,dy^2
\]
where
\[
	g_{yy}:[0,\infty) \to [0,\infty)
\]
is a smooth function which is positive over $(0,\infty)$.
Let
\[
	V: [0, \infty) \to [0,\infty)
\]
be another smooth function which is positive over $(0, \infty)$, and assume that $V'$ is positive over $(0,\infty)$.
Fix constants
\[
	r_1 > r_2 \ge 0
\]
If $V(0)=0$, then assume that $r_2>0$.

Let $L>0$ be finite and let
\[
	\alpha:[0,L] \to \overline{\R^2_+}
\]
be a piecewise smooth curve.
Assume that $\alpha(t)$ is in $\R^2_+$ for all $t$ in $[0,L)$.
Also assume that $\alpha$ is parametrized by arc length with respect to $g$.
Write $\alpha=(F_\alpha,G_\alpha)$, i.e. let $F_\alpha$ and $G_\alpha$ be the component functions of $\alpha$.
Assume that
\[
	\alpha(0) = (r_1,0)
\]
and
\[
	F_\alpha(L) = r_2
\]
Define a curve
\[
	\omega:[0,r_1-r_2] \to \overline{\R^2_+}
\]
by
\[
	\omega(t) = (r_1 - t, 0)
\]

Let $V_\alpha=V \circ F_\alpha$.
For a non-negative real number $k$ and a positive integer $n$, define
\[
	\lambda_{k,n}(\alpha) = \min_W \max_{w \in W} \left\{ \frac{ \int_0^L |w'|^2 V_\alpha + \frac{k^2 |w|^2}{V_\alpha} \,dt}{ \int_0^L |w|^2 V_\alpha \,dt} \right\}
\]
Here the minimum is taken over all $n$-dimensional subspaces $W$ of $C_0^\infty(0,L)$.
Define $\lambda_{k,n}(\omega)$ similarly.

\begin{Lemma}
\label{kngen}
If $\alpha$ is not equal to $\omega$, then for any non-negative real number $k$ and any positive integer $n$,
\[
	\lambda_{k,n}(\alpha) < \lambda_{k,n}(\omega)
\]
\end{Lemma}

Theorem \ref{annulus} and Theorem \ref{helicoid} follow immediately from Lemma \ref{knannulus} and Lemma \ref{knheli}, respectively.
These lemmas are special cases of Lemma \ref{kngen}.
The rest of the article is a proof of Lemma \ref{kngen}.
We remark that considering non-integer values of $k$ allows for more applications, including an analogue of Theorem \ref{annulus} in higher dimensions.

Note that the eigenvalues $\lambda_{k,n}$ satisfy the following monotonicity property.
If $\beta$ is a curve derived from $\alpha$ by restricting to a proper subinterval of $[0,L]$, then for every non-negative real number $k$ and every positive integer $n$,
\[
	\lambda_{k,n}(\alpha) < \lambda_{k,n}(\beta)
\]
Therefore, to prove Lemma \ref{kngen}, it suffices to consider the case where, for all $t$ in $(0,L)$,
\[
	r_2 < F_\alpha(t) < r_1
\]

\section{Reduction to a Special Case}

In this section, we begin the proof of Lemma \ref{kngen}.
We reduce the problem to the special case where $n=1$.
Then the special case $n=1$ is established in the next section.
The material in this section closely follows the argument in \cite{Adisc}.
We first extend the definition of the eigenvalues $\lambda_{k,n}$ to Lipschitz curves.

\begin{Definition}
\label{eiglip}
Let $[a,b]$ be a finite, closed interval and let
\[
	\gamma:[a,b] \to \overline{\R^2_+}
\]
be Lipschitz.
Assume that $\gamma(t)$ is in $\R^2_+$ for all $t$ in $[a,b)$.
Write
\[
	\gamma=(F_\gamma, G_\gamma)
\]
i.e. let $F_\gamma$ and $G_\gamma$ be the component functions of $\gamma$.
Let
\[
	V_\gamma=V \circ F_\gamma
\]
Let $\Lip_0(a,b)$ be the set of real-valued Lipschitz continuous functions on $[a,b]$ which vanish at the endpoints.
For a non-negative real number $k$ and a positive integer $n$, define
\[
	\lambda_{k,n}(\gamma) = \inf_W \max_{w \in W} \Bigg\{ \frac{ \int_a^b \frac{|w'|^2 V_\gamma}{|\gamma '|} + \frac{k^2 |w|^2 | \gamma' |}{V_\gamma} \,dt}{ \int_a^b |w|^2 V_\gamma | \gamma' | \,dt} \Bigg\}
\]
Here the infimum is taken over all $n$-dimensional subspaces $W$ of $\Lip_0(a,b)$.
\end{Definition}

The definition of the eigenvalues $\lambda_{k,n}(\gamma)$ depends on two functions, $V_\gamma$ and $|\gamma'|$ which derive from a curve $\gamma$.
It will be useful to extend the definition to functions $v$ and $\sigma$ which do not necessarily derive from a curve.

\begin{Definition}
\label{eigvsig}
Let $[a,b]$ be a finite interval, and let
\[
	v:[a,b] \to [0,\infty)
\]
and
\[
	\sigma:[a,b] \to [0,\infty)
\]
be functions in $L^\infty(a,b)$.
For a non-negative real number $k$ and a positive integer $n$, define
\[
	\lambda_{k,n}(v,\sigma) = \inf_W \max_{w \in W} \Bigg\{ \frac{ \int_a^b \frac{|w'|^2 v}{\sigma} + \frac{k^2 |w|^2 \sigma}{v} \,dt}{ \int_a^b |w|^2 v \sigma \,dt} \Bigg\}
\]
Here the infimum is taken over all $n$-dimensional subspaces $W$ of $\Lip_0(a,b)$.
\end{Definition}

For functions
\[
	v:[a,b] \to [0,\infty)
\]
and
\[
	\sigma:[a,b] \to [0,\infty)
\]
which are in $L^\infty(a,b)$, let $H_0^1(v,\sigma,k)$ be the set of continuous functions
\[
	w:[a,b] \to \R
\]
which vanish at the endpoints and have a weak derivative such that
\[
	\int_a^b \frac{|w'|^2 v}{\sigma} + \frac{k^2 w^2 \sigma}{v} \,dt < \infty
\]

\begin{Lemma}
\label{efex}
Let
\[
	v:[a,b] \to (0,\infty)
\]
and
\[
	\sigma:[a,b] \to (0,\infty)
\]
be functions in $L^\infty(a,b)$.
Assume that $1/v$ and $1/\sigma$ are in $L^\infty(a,b)$ as well.
Let $k$ be a non-negative real number.
Then there are functions
\[
	\phi_{k,1}, \phi_{k,2}, \phi_{k,3}, \ldots
\]
which form an orthonormal basis of $H_0^1(v,\sigma,k)$ such that, for any positive integer $n$,
\[
	\lambda_{k,n}(v,\sigma) = \frac{ \int_a^b \frac{|\phi_{k,n}'|^2 v}{\sigma} + \frac{k^2 |\phi_{k,n}|^2 \sigma}{v} \,dt}{ \int_a^b |\phi_{k,n}|^2 v \sigma \,dt}
\]
Each function $\phi_{k,n}$ has exactly $n-1$ roots in $(a,b)$ and satisfies the following equation weakly:
\[
	\bigg( \frac{v \phi_{k,n}'}{\sigma} \bigg)' = \frac{k^2 \sigma \phi_{k,n}}{v} - \lambda_{k,n}(\psi) v \sigma \phi_{k,n}
\]
Also
\[
	\lambda_{k,1}(v,\sigma) < \lambda_{k,2}(v,\sigma) < \lambda_{k,3}(v,\sigma) < \ldots
\]
\end{Lemma}

We omit the proof of this lemma, which is standard.
For details, we refer to Zettl \cite[Theorem 10.12.1]{Z} and Gilbarg and Trudinger \cite[Section 8.12]{GT}.

We refer to the functions $\phi_{k,n}$ given by Lemma \ref{efex} as the eigenfunctions corresponding to $v$ and $\sigma$.
If there is a curve $\gamma$ such that $v=V_\gamma$ and $\sigma=|\gamma'|$, then we also refer to the functions $\phi_{k,n}$ as the eigenfunctions corresponding to $\gamma$.

Fix a non-negative real number $k$ and a positive integer $n$.
Let
\[
	\mu_0 = \frac{k}{\sqrt{\lambda_{k,n}(\omega)}}
\]
Let $\Phi=\Phi_{k,n}$ be the eigenfunction corresponding to $\omega$ given by Lemma \ref{efex}.

\begin{Lemma}
\label{muV}
If $z$ is a root of $\Phi$ in $[0,r_1-r_2)$, then
\[
	\mu_0 < V_\omega(z)
\]
In particular,
\[
	\mu_0 < V(r_1)
\]
\end{Lemma}

\begin{proof}
Write $\omega=(F_\omega,G_\omega)$ and let $V_\omega=V \circ F_\omega$.
It follows from Lemma \ref{efex} that
\[
	\lambda_{k,n}(\omega) = \frac{ \int_z^{r_1-r_2} |\Phi'|^2 V_\omega + \frac{k^2 |\Phi|^2}{V_\omega} \,dt}{ \int_z^{r_1-r_2} |\Phi|^2 V_\omega \,dt}
\]
Since $\Phi$ is non-constant and $V$ is monotonic, this implies that
\[
	\lambda_{k,n}(\omega) > \frac{k^2}{[V_\omega(z)]^2}
\]
That is, $\mu_0 < V_\omega(z)$.
Note that the case $z=0$ yields $\mu_0 < V_\omega(r_1)$.
\end{proof}

Write $\alpha=(F_\alpha,G_\alpha)$, i.e. let $F_\alpha$ and $G_\alpha$ be the component functions of $\alpha$.
Let $V_\alpha=V \circ F_\alpha$.
If $\mu_0 < V(r_2)$, then define $P_0=L$.
Otherwise, define
\[
	P_0 = \min \bigg\{ t \in [0,L] : V_\alpha(t)= \mu_0 \bigg\}
\]
In light of Lemma \ref{muV}, the value $P_0$ is well-defined and positive.
Define
\[
	\beta_0:[0,L] \to \overline{\R^2_+}
\]
to be a piecewise smooth function such that
\[
	\beta_0(0)=(r_1,0)
\]
and
\[
	\beta_0'(t) =
	\begin{cases}
		(F_\alpha'(t), 0) & t \in [0,P_0) \\
		(F_\alpha'(t), G_\alpha'(t) & t \in (P_0,L] \\
	\end{cases}
\]
Note that $|\beta_0'| \le |\alpha'|$ over $[0,P_0)$, and $|\beta_0'|=|\alpha'|$ over $(P_0,L]$.

\begin{Lemma}
\label{a1}
Assume $\alpha$ is not equal to $\beta_0$ and $\lambda_{k,n}(\alpha) \ge \lambda_{k,n}(\omega)$.
Then
\[
	\lambda_{k,n}(\alpha) < \lambda_{k,n}(\beta_0)
\]
\end{Lemma}

\begin{proof}
Fix a number $p$ in $(0,1)$.
Define
\[
	\alpha_p:[0,L] \to \overline{\R^2_+}
\]
to be a regular piecewise smooth curve such that
\[
	\alpha_p(0)=(r_1,0)
\]
and
\[
	\alpha_p'(t) =
	\begin{cases}
		(F_\alpha'(t), p G_\alpha'(t)) & t \in [0,P_0) \\
		(F_\alpha'(t), G_\alpha'(t) & t \in (P_0,L] \\
	\end{cases}
\]
We first show that
\[
	\lambda_{k,n}(\alpha) < \lambda_{k,n}(\alpha_p)
\]
By Lemma \ref{efex}, there is an $n$-dimensional subspace $W_p$ of $\Lip_0(0,L)$ such that
\[
	\lambda_{k,n}(\alpha_p) = \max_{w \in W_p} \frac{\int_0^L \frac{|w'|^2 V_{\alpha}}{|\alpha_p'|} + \frac{k^2 |w|^2 | \alpha_p'|}{V_{\alpha}} \,dt}{\int_0^L |w|^2 V_\alpha |\alpha_p'| \,dt}
\]
Moreover, the maximum over $W_p$ is only attained by scalar multiples of a function $\phi_{k,n}$ which has exactly $n-1$ roots in $(0,L)$.
Let $v$ be a function in $W_p$ such that
\[
	\frac{\int_0^L \frac{|v'|^2 V_{\alpha}}{|\alpha'|} + \frac{k^2 |v|^2 | \alpha'|}{V_{\alpha}} \,dt}{\int_0^L |v|^2 V_\alpha |\alpha'| \,dt}	
	= \max_{w \in W_p} \frac{\int_0^L \frac{|w'|^2 V_{\alpha}}{|\alpha'|} + \frac{k^2 |w|^2 | \alpha'|}{V_{\alpha}} \,dt}{\int_0^L |w|^2 V_\alpha |\alpha'| \,dt}
\]
Note this quantity is at least $\lambda_{k,n}(\alpha)$, which is at least $\lambda_{k,n}(\omega)$.
It follows that
\[
	\frac{\int_0^L \frac{|v'|^2 V_{\alpha}}{|\alpha'|} + \frac{k^2 |v|^2 | \alpha'|}{V_{\alpha}} \,dt}{\int_0^L |v|^2 V_\alpha |\alpha'| \,dt}	
	\le \frac{\int_0^L \frac{|v'|^2 V_{\alpha}}{|\alpha_p'|} + \frac{k^2 |v|^2 | \alpha_p'|}{V_{\alpha}} \,dt}{\int_0^L |v|^2 V_\alpha |\alpha_p'| \,dt}
\]
Moreover, if equality holds, then $v$ must vanish on a set of positive measure.
In this case, $v$ is not a multiple of $\phi_{k,n}$, so
\[
	\frac{\int_0^L \frac{|v'|^2 V_{\alpha}}{|\alpha_p'|} + \frac{k^2 |v|^2 | \alpha_p'|}{V_{\alpha}} \,dt}{\int_0^L |v|^2 V_\alpha |\alpha_p'| \,dt}
	< \lambda_{k,n}(\alpha_p)
\]
In either case, we obtain
\[
	\lambda_{k,n}(\alpha) \le \frac{\int_0^L \frac{|v'|^2 V_{\alpha}}{|\alpha'|} + \frac{k^2 |v|^2 | \alpha'|}{V_{\alpha}} \,dt}{\int_0^L |v|^2 V_\alpha |\alpha'| \,dt} < \lambda_{k,n}(\alpha_p)
\]

Now we repeat the argument to show that
\[
	\lambda_{k,n}(\alpha_p) \le \lambda_{k,n}(\beta_0)
\]
Let $\epsilon>0$.
There is an $n$-dimensional subspace $W$ of $\Lip_0(0,L)$ such that
\[
	 \max_{w \in W} \frac{\int_0^L \frac{|w'|^2 V_{\alpha}}{|\beta_0'|} + \frac{k^2 |w|^2 | \beta_0'|}{V_{\alpha}} \,dt}{\int_0^L |w|^2 V_\alpha |\beta_0'| \,dt}
	 < \lambda_{k,n}(\beta_0) + \epsilon
\]
Let $u$ be a function in $W$ such that
\[
	\frac{\int_0^L \frac{|u'|^2 V_{\alpha}}{|\alpha_p'|} + \frac{k^2 |u|^2 | \alpha_p'|}{V_{\alpha}} \,dt}{\int_0^L |u|^2 V_\alpha |\alpha_p'| \,dt}	
	= \max_{w \in W} \frac{\int_0^L \frac{|w'|^2 V_{\alpha}}{|\alpha_p'|} + \frac{k^2 |w|^2 | \alpha_p'|}{V_{\alpha}} \,dt}{\int_0^L |w|^2 V_\alpha |\alpha_p'| \,dt}
\]
Note this quantity is at least $\lambda_{k,n}(\alpha_p)$, which is at least $\lambda_{k,n}(\omega)$.
As above, it follows that
\[
	\frac{\int_0^L \frac{|u'|^2 V_{\alpha}}{|\alpha_p'|} + \frac{k^2 |u|^2 | \alpha_p'|}{V_{\alpha}} \,dt}{\int_0^L |u|^2 V_\alpha |\alpha_p'| \,dt}
	\le \frac{\int_0^L \frac{|u'|^2 V_{\alpha}}{|\beta_0'|} + \frac{k^2 |u|^2 | \beta_0'|}{V_{\alpha}} \,dt}{\int_0^L |u|^2 V_\alpha |\beta_0'| \,dt}
\]
Now we obtain
\[
	\lambda_{k,n}(\alpha_p) \le \lambda_{k,n}(\beta_0) + \epsilon
\]
Therefore,
\[
	\lambda_{k,n}(\alpha) < \lambda_{k,n}(\beta_0)
\]

\end{proof}

Write $\beta_0=(F_{\beta_0},G_{\beta_0})$.
Define
\[
	F_{\gamma_0}:[0,L] \to [0,\infty)
\]
by
\[
	F_{\gamma_0}(t) =
	\begin{cases}
		\min \{ F_{\beta_0}(s) : s \in [0,t] \} & t \in [0,P_0] \\
		F_{\beta_0}(t) & t \in [P_0,L] \\
	\end{cases}
\]
Let $G_{\gamma_0}=G_{\beta_0}$ and define
\[
	\gamma_0=(F_{\gamma_0}, G_{\gamma_0})
\]
Note that
\[
	\gamma_0:[0,L] \to \overline{\R^2_+}
\]
is Lipschitz.

\begin{Lemma}
\label{12}
Assume $\lambda_{k,n}(\beta_0) \ge \lambda_{k,n}(\omega)$.
Then
\[
	\lambda_{k,n}(\beta_0) \le \lambda_{k,n}(\gamma_0)
\]
\end{Lemma}

\begin{proof}
Define
\[
	I = \bigg\{ t \in [0,P_0] : F_{\beta_0}(t) \neq F_{\gamma_0}(t) \bigg\}
\]
By the Riesz sunrise lemma, there are disjoint open intervals $I_1, I_2, I_3, \ldots$ such that
\[
	I = \bigcup_j I_j
\]
and $F_{\gamma_0}$ is constant over each interval.
Suppose
\[
	\lambda_{k,n}(\beta_0) > \lambda_{k,n}(\gamma_0)
\]
Then there is an $n$-dimensional subspace $W$ of $\Lip_0(0,L)$ such that
\[
	\max_{w \in W} \frac{\int_0^L \frac{|w'|^2 V_{\gamma_0}}{|\gamma_0'|} + \frac{k^2 |w|^2 | \gamma_0'|}{V_{\gamma_0}} \,dt}{\int_0^L |w|^2 V_{\gamma_0} |\gamma_0'| \,dt} < \lambda_{k,n}(\beta_0)
\]
Note that $\gamma_0'$ is zero over the intervals $I_1, I_2, I_3, \ldots$, so every function $w$ in $W$ is constant over each of these intervals.
Let
\[
	J=[0,L] \setminus I
\]
The isolated points of $J$ are countable, so at almost every point in $J$, the curve $\gamma_0$ is differentiable with $\gamma_0'=\beta_0'$.
If $w$ is a non-zero function in $W$, then $w$ cannot vanish identically on $J$, and
\[
	\frac{\int_J \frac{|w'|^2 V_{\beta_0}}{|\beta_0'|} + \frac{k^2 |w|^2 | \beta_0'|}{V_{\beta_0}} \,dt}{\int_J |w|^2 V_{\beta_0} |\beta_0'| \,dt}
	= \frac{\int_0^L \frac{|w'|^2 V_{\gamma_0}}{|\gamma_0'|} + \frac{k^2 |w|^2 | \gamma_0'|}{V_{\gamma_0}} \,dt}{\int_0^L |w|^2 V_{\gamma_0} |\gamma_0'| \,dt} < \lambda_{k,n}(\beta_0)
\]
Also for every $w$ in $W$,
\[
\begin{split}
	\int_I \frac{|w'|^2 V_{\beta_0}}{|\beta_0'|} + \frac{k^2 |w|^2 | \beta_0'|}{V_{\beta_0}} \,dt
	&= \int_I \frac{k^2 |w|^2 | \beta_0'|}{V_{\beta_0}} \,dt \\
	& \le \lambda_{k,n}(\omega) \int_I |w|^2 V_{\beta_0} | \beta_0'| \,dt \\
\end{split} 
\]
It follows that
\[
	\max_{w \in W} \frac{\int_0^L \frac{|w'|^2 V_{\beta_0}}{|\beta_0'|} + \frac{k^2 |w|^2 | \beta_0'|}{V_{\beta_0}} \,dt}{\int_0^L |w|^2 V_{\beta_0} |\beta_0'| \,dt} < \lambda_{k,n}(\beta_0)
\]
This is a contradiction.
\end{proof}

Let $L_0$ be the length of $\gamma_0$.
Define $\ell_0:[0,L] \to [0,L_0]$ by
\[
	\ell_0(t) = \int_0^t |\gamma'(u)| \,du
\]
Define
\[
	\rho_0:[0,L_0] \to [0,L]
\]
by
\[
	\rho_0(s) = \min \Big\{ t \in [0,L] : \ell_0(t) = s \Big\}
\]
The function $\rho_0$ need not be continuous, but the curve
\[
	\chi_0=\gamma_0 \circ \rho_0
\]
is piecewise smooth, and for all $t$ in $[0,L]$,
\[
	\chi_0(\ell_0(t)) = \gamma_0(t)
\]
Moreover $\chi_0$ is parametrized by arc length.
Note that for all $t$ in $[0,V_\omega^{-1}(\mu_0)]$,
\[
	\chi_0(t)=\omega(t)
\]

\begin{Lemma}
\label{23}
This reparametrization satisfies
\[
	\lambda_{k,n}(\gamma_0) \le \lambda_{k,n}(\chi_0)
\]
\end{Lemma}

\begin{proof}
Write $\gamma_0=(F_{\gamma_0}, G_{\gamma_0})$ and $\chi_0=(F_{\chi_0}, G_{\chi_0})$.
Let $V_{\gamma_0}=V \circ F_{\gamma_0}$ and $V_{\chi_0} = V \circ F_{\chi_0}$.
Let $w$ be a function in $\Lip_0(0,L_0)$ such that
\[
	\frac{\int_0^{L_0} \frac{|w'|^2 V_{\chi_0}}{|\chi_0'|} + \frac{k^2 |w|^2 | \chi_0'|}{V_{\chi_0}} \,dt}{\int_0^{L_0} |w|^2 V_{\chi_0} |\chi_0'| \,dt} < \infty
\]
Define $v=w \circ \ell_0$.
Then $v$ is in $\Lip_0(0,L)$, and changing variables yields
\[
	\frac{\int_0^L \frac{|v'|^2 V_{\gamma_0}}{|\gamma_0'|} + \frac{k^2 |v|^2 | \gamma_0'|}{V_{\gamma_0}} \,dt}{\int_0^L |v|^2 V_{\gamma_0} |\gamma_0'| \,dt}
	= \frac{\int_0^{L_0} \frac{|w'|^2 V_{\chi_0}}{|\chi_0'|} + \frac{k^2 |w|^2 | \chi_0'|}{V_{\chi_0}} \,dt}{\int_0^{L_0} |w|^2 V_{\chi_0} |\chi_0'| \,dt}
\]
It follows that
\[
	\lambda_{k,n}(\gamma_0) \le \lambda_{k,n}(\chi_0)
\]
\end{proof}

Note that if $\chi_0=\omega$, then Lemma \ref{kngen} now follows immediately.
In particular, this holds if $\mu_0 \le V(r_2)$.

\begin{proof}[Proof of Lemma 2.3 for the case $\chi_0=\omega$]
Suppose $\alpha$ is not equal to $\omega$ and
\[
	\lambda_{k,n}(\alpha) \ge \lambda_{k,n}(\omega)
\]
Then
 $\alpha$ is not equal to $\beta_0$, so by Lemmas \ref{a1}, \ref{12}, and \ref{23},
\[
	\lambda_{k,n}(\alpha) < \lambda_{k,n}(\beta_0) \le \lambda_{k,n}(\gamma_0)  \le \lambda_{k,n}(\chi_0) = \lambda_{k,n}(\omega)
\]
\end{proof}

We conclude this section by showing that, in order to prove Lemma \ref{kngen}, it suffices to consider the case $n=1$.

\begin{proof}[Proof of Lemma 2.3, assuming the case $n=1$ holds]
Note that we may also assume that $\chi_0 \neq \omega$.
Suppose that
\[
	\lambda_{k,n}(\alpha) \ge \lambda_{k,n}(\omega)
\]
By Lemmas \ref{a1}, \ref{12}, and \ref{23},
\[
	\lambda_{k,n}(\alpha) \le \lambda_{k,n}(\chi_0)
\]
Let $\Phi=\Phi_{k,n}$ be the eigenfunction given by Lemma \ref{efex} corresponding to the curve $\omega$.
Let $z_0$ be the largest root of $\Phi$ in $[0,r_1-r_2)$.
By Lemma \ref{muV},
\[
	z_0 < V_\omega^{-1}(\mu_0)
\]
Let $\phi=\phi_{k,n}$ be the eigenfunction given by Lemma \ref{efex} corresponding to the curve $\chi_0$.
Let $\tau_0$ be the largest root of $\phi$.
Note that $\chi_0$ agrees with $\omega$ over $[0, V_\omega^{-1}(\mu_0)]$, and
\[
	\lambda_{k,n}(\chi_0) \ge \lambda_{k,n}(\omega)
\]
By the Sturm comparison theorem,
\[
	\tau_0 \le z_0
\]
Define curves $\alpha_0$ and $\omega_0$ by
\[
	\alpha_0 = \chi_0 \Big|_{[\tau_0,L]}
\]
and
\[
	\omega_0 = \omega \Big|_{[\tau_0,L]}
\]
It follows from Lemma \ref{efex} that
\[
	\lambda_{k,n}(\chi_0) = \lambda_{k,1}(\alpha_0)
\]
Also
\[
	\lambda_{k,n}(\omega) \ge \lambda_{k,1}(\omega_0)
\]
We have $\alpha_0 \neq \omega_0$, because we are assuming that $\chi_0 \neq \omega$.
By the assumption that Lemma \ref{kngen} holds for the case $n=1$,
\[
	\lambda_{k,1}(\alpha_0) < \lambda_{k,1}(\omega_0)
\]
Now,
\[
	\lambda_{k,n}(\alpha) \le \lambda_{k,n}(\chi_0) = \lambda_{k,1}(\alpha_0) < \lambda_{k,1}(\omega_0) \le \lambda_{k,n}(\omega) 
\]
\end{proof}

\section{Proof of the Special Case}

In this section we prove Lemma \ref{kngen} for the case $n=1$.
We also assume $\chi_0 \neq \omega$.
This implies that
\[
	\mu_0 > V(r_2)
\]
Let $\Phi=\Phi_{k,1}$ be the eigenfunction given by Lemma \ref{efex} corresponding to the curve $\omega$.
Write $\omega=(F_\omega, G_\omega)$ and let $V_\omega=V \circ F_\omega$.
Note that $V_\omega'$ is negative over $[0,r_1-r_2)$.
Define a function
\[
	Y:(0,r_1-r_2) \to \R
\]
by
\[
	Y(t) = \frac{V_\omega \Phi'}{\Phi}
\]

\begin{Lemma}
\label{Y}
Over $(0, r_1-r_2)$, we have $Y' < 0$ and
\[
	Y^2 > k^2 - \lambda_{k,1}(\omega) V_\omega^2
\]
Moreover, $V_\omega^{-1}(\mu_0)$ is in $(0, r_1-r_2)$, and $Y < 0$ over $[V_\omega^{-1}(\mu_0), r_1-r_2)$.
\end{Lemma}

\begin{proof}
By Lemma \ref{efex}, the eigenfunction $\Phi$ is smooth and satisfies
\[
	(V_\omega \Phi')' = \frac{k^2 \Phi}{V_\omega} - \lambda_{k,1}(\omega) V_\omega \Phi
\]
Therefore $Y$ is smooth over $(0, r_1-r_2)$ and satisfies
\[
	Y' = \frac{k^2}{V_\omega} - \lambda_{k,1}(\omega) V_\omega - \frac{Y^2}{V_\omega}
\]
We first show that $Y$ has no critical points.
To see this, suppose $t_0$ is a critical point of $Y$.
Then
\[
	[Y(t_0)]^2 < k^2
\]
This yields
\[
	Y''(t_0)= -V_\omega'(t_0) \bigg( \frac{k^2}{[V_\omega(t_0)]^2}+\lambda_{k,1}(\omega) - \frac{[Y(t_0)]^2}{[V_\omega(t_0)]^2} \bigg) > 0
\]
That is, every critical point of $Y$ is a local minimum.
Note that
\[
	\lim_{t \to r_1-r_2} Y(t) = -\infty
\]
It follows that $Y$ has no critical points and $Y'<0$ over $(0,r_1-r_2)$.

Lemma \ref{muV} and the assumption that $\mu_0 > V(r_2)$ imply that $V_\omega^{-1}(\mu_0)$ is in $(0,r_1-r_2)$.
We will show that, over $(0,r_1-r_2)$,
\[
	Y^2 > k^2 - \lambda_{k,1}(\omega) V_\omega^2
\]
This inequality implies that $Y$ is non-vanishing over $[V_\omega^{-1}(\mu_0)), r_1-r_2)$, i.e. $Y<0$ over $[V_\omega^{-1}(\mu_0),r_1-r_2)$.
To prove the inequality, it suffices to show that, over $[0,r_1-r_2]$,
\[
	\lambda_{k,1}(\omega) V_\omega^2 \Phi^2 + V_\omega^2 (\Phi')^2 - k^2 \Phi^2 > 0
\]
This is obvious at the endpoints, where $\Phi$ vanishes and $\Phi'$ does not.
It also follows in the interior, because over $(0,r_1-r_2)$,
\[
	\frac{d}{dt} \bigg[ \lambda_{k,1}(\omega) V_\omega^2 \Phi^2 + V_\omega^2 (\Phi')^2 - k^2 \Phi^2 \bigg] = 2 \lambda_{k,1}(\omega) V_\omega V_\omega' \Phi^2 < 0
\]
\end{proof}

Recursively define sequences $\{\mu_m\}$ and $\{ y_m \}$ as follows.
Recall $\mu_0$ is defined by
\[
	\mu_0 = \frac{k}{\sqrt{\lambda_{k,1}(\omega)}}
\]
Let $m$ be a positive integer.
Assume that $\mu_{m-1}$ is defined and
\[
	V(r_2) \le \mu_{m-1} \le \mu_0
\]
If $\mu_{m-1}=V(r_2)$, then define $y_{m-1}=-\infty$.
Otherwise, define
\[
	y_{m-1} = Y(V_\omega^{-1}(\mu_{m-1}))
\]
If $y_{m-1}^2 > k^2$, define $\mu_m=V(r_2)$.
Otherwise, define
\[
	\mu_m = \max \bigg\{ V(r_2), \sqrt{\frac{k^2 - y_{m-1}^2}{\lambda_{k,1}(\omega)}} \bigg\}
\]

\begin{Lemma}
\label{mumono}
The sequence $\{ \mu_m \}$ is monotonically decreasing.
Moreover, if $M$ is large, then $\mu_M=V(r_2)$.
\end{Lemma}

\begin{proof}
By Lemma \ref{Y}, we have $y_0 < 0$.
By assumption,
\[
	\mu_0>V(r_2)
\]
It is immediate that
\[
	\mu_1 < \mu_0
\]
Fix a positive integer $m$ and assume that
\[
	\mu_m \le \mu_{m-1} \le \mu_0
\]
By Lemma \ref{Y},
\[
	y_m \le y_{m-1} \le 0
\]
This immediately implies that
\[
	\mu_{m+1} \le \mu_m
\]
Therefore, the sequence $\{ \mu_m \}$ is decreasing.

Now suppose that, for all $m$,
\[
	\mu_m > V(r_2)
\]
In this case, for all $m$,
\[
	y_m^2 \le k^2
\]
for all $m$.
The sequence $\{ \mu_m \}$ converges to some number $\mu_*$ in $[V(r_2),\mu_0]$.
Also, the sequence $\{ y_m \}$ converges to a finite number $y_*$ with $y_*^2 \le k^2$.
In fact, $\mu_*>V(r_2)$ and
\[
	y_* = Y(V_\omega^{-1}(\mu_*))
\]
Moreover,
\[
	\lambda_{k,1}(\omega) \mu_*^2 = k^2- y_*^2
\]
Therefore,
\[
	(Y(V_\omega^{-1}(\mu_*))^2 = k^2 - \lambda_{k,1}(\omega) \mu_*^2
\]
By Lemma \ref{Y}, this is a contradiction.
\end{proof}

\begin{Definition}
For each $\mu$ in $[V(r_2),V(r_1))$, let $\Gamma_\mu$ be the set of curves
\[
	\chi:[0,D] \to \overline{\R^2_+}
\]
which satisfy the following four properties.
First, $D$ is a finite number such that
\[
	D \ge r_1-r_2
\]
Second, $\chi$ is piecewise smooth and parametrized by arc length.
Third, if $\chi=(F_\chi, G_\chi)$, i.e. if $F_\chi$ and $G_\chi$ are the component functions of $\chi$, then
\[
	F_\chi(D)=r_2
\]
and, for all $t$ in $(0,D)$,
\[
	r_2 < F_\chi(t) < r_1
\]
Fourth, for all $t$ in $[0, V_\omega^{-1}(\mu)]$,
\[
	\chi(t)=\omega(t)
\]
\end{Definition}

Note that $\chi_0$ is in $\Gamma_{\mu_0}$.
Also, the only curve in $\Gamma_{V(r_2)}$ is $\omega$.

\begin{Lemma}
\label{mm1}
Let $m$ be a positive integer.
Assume that there is a curve $\chi_{m-1}$ in $\Gamma_{\mu_{m-1}}$ such that
\[
	\lambda_{k,1}(\chi_{m-1}) \ge \lambda_{k,1}(\omega)
\]
Then there is a curve $\chi_m$ in $\Gamma_{\mu_m}$ such that either $\chi_m=\chi_{m-1}$ or
\[
	\lambda_{k,1}(\chi_m) > \lambda_{k,1}(\chi_{m-1})
\]
\end{Lemma}

Assuming the lemma for now, we use it to prove Lemma \ref{kngen} for the case $n=1$.

\begin{proof}[Proof of Lemma 2.3 for the case $n=1$]
Suppose $\lambda_{k,1}(\alpha) \ge \lambda_{k,1}(\omega)$.
By Lemmas \ref{a1}, \ref{12}, and \ref{23},
\[
	\lambda_{k,1}(\alpha) \le \lambda_{k,1}(\chi_0)
\]
Additionally, if equality holds, then $\alpha=\chi_0$.
Use Lemma \ref{mm1} to recursively define a sequence of curves $\{ \chi_m \}$ such that, for every positive integer $m$, the curve $\chi_m$ is in $\Gamma_{\mu_m}$, and either $\chi_m=\chi_{m-1}$ or
\[
	\lambda_{k,1}(\chi_{m-1}) < \lambda_{k,1}(\chi_m)
\]
By Lemma \ref{mumono}, there is a positive integer $M$ such that
\[
	\mu_M=V(r_2)
\]
Then $\chi_M$ is in $\Gamma_{V(r_2)}$ which implies that
\[
	\chi_M=\omega
\]
Now
\[
	\lambda_{k,1}(\alpha) \le \lambda_{k,1}(\chi_0) \le \ldots \le \lambda_{k,1}(\chi_M)=\lambda_{k,1}(\omega)
\]
By assumption,
\[
	\lambda_{k,1}(\alpha) \ge \lambda_{k,1}(\omega)
\]
Therefore
\[
	\lambda_{k,1}(\alpha) = \lambda_{k,1}(\chi_0) = \ldots = \lambda_{k,1}(\chi_M)=\lambda_{k,1}(\omega)
\]
This implies that
\[
	\alpha=\chi_0=\ldots=\chi_M=\omega
\]
\end{proof}

It remains to prove Lemma \ref{mm1}.
The argument is based on the following lemma.
We remark that the statement of this lemma uses Definition \ref{eigvsig}.

\begin{Lemma}
\label{efes}
Let $Q \ge r_1-r_2$ be a finite number.
Let
\[
	v:[0,Q] \mapsto (0,\infty)
\]
and
\[
	\sigma:[0,Q] \mapsto (0,\infty)
\]
be continuous functions.
Let $m$ be a positive integer.
Assume that, for $t$ in $[0,V_\omega^{-1}(\mu_{m-1})]$,
\[
	v(t)=V_\omega(t)
\]
and, for $t$ in $[0,V_\omega^{-1}(\mu_{m-1})]$,
\[
	\sigma(t)=1
\]
Also assume that $v(Q)=V(r_2)$ and that
\[
	\lambda_{k,1}(v,\sigma) \ge \lambda_{k,1}(\omega)
\]
Let $\phi=\phi_{k,1}$ be the eigenfunction given by Lemma \ref{efex}, corresponding to $v$ and $\sigma$.
Define
\[
	P_m = \min \bigg\{ t \in [0,Q] : v(t)= \mu_m \bigg\}
\]
Then, over $[0,P_m)$,
\[
	- \frac{|\phi'|^2 v^2}{\sigma^2} + k^2 |\phi|^2 - \lambda_{k,1}(v,\sigma) |\phi|^2 v^2 < 0
\]
\end{Lemma}

\begin{proof}
The inequality is obvious over $[0, V_\omega^{-1}(\mu_0)]$.
We will prove it holds over $(V_\omega^{-1}(\mu_0),P_m)$.
By Lemma \ref{efex}, the eigenfunction $\phi$ is continuously differentiable over $[0,Q]$ and
\[
	\bigg( \frac{ v \phi' }{ \sigma } \bigg)' = \frac{k^2 \sigma \phi}{v} - \lambda_{k,n}(v,\sigma) v \sigma \phi
\]
Recall $\Phi=\Phi_{k,1}$ is the eigenfunction corresponding to $\omega$.
We may assume that $\phi$ and $\Phi$ are positive over $(0,Q)$ and $(0,r_2-r_1)$, respectively.

Define
\[
	X:(0,Q) \to \R
\]
by
\[
	X(t) = \frac {v \phi'}{\sigma \phi}
\]
Then $X$ is differentiable over $(0,Q)$, and
\[
	X' = \frac{k^2 \sigma}{v} - \lambda_{k,1}(v,\sigma) v \sigma - \frac {\sigma X^2}{v}
\]
To prove the lemma, it suffices to show that $X'<0$ over $[V_\omega^{-1}(\mu_0),P_m)$,

Define an interval
\[
	I = [V_\omega^{-1}(\mu_0),V_\omega^{-1}(\mu_{m-1})] \cap [V_\omega^{-1}(\mu_0),P_m)
\]
First we show that $X'<0$ holds over $I$.
We do this by comparing to $Y$, which was defined by
\[
	Y = \frac{V_\omega \Phi'}{\Phi}
\]
Note that, over $(0,r_1-r_2)$,
\[
	Y' = \frac{k^2}{V_\omega} - \lambda_{k,1}(\omega) V_\omega - \frac{Y^2}{V_\omega}
\]
For $t$ in $(0, V_\omega^{-1}(\mu_{m-1})]$,
\[
\begin{split}
	0 &\le \Big( \lambda_{k,1}(v,\sigma) - \lambda_{k,1}(\omega) \Big) \int_{0}^t V_\omega \phi \Phi \\
	& = \int_{0}^t (V_\omega \Phi')' \, \phi -(V_\omega \phi')' \, \Phi \\
	&= V_\omega(t) \Phi'(t) \phi(t) - V_\omega(t) \phi'(t) \Phi(t) \\
\end{split}
\]
Therefore, over $I$,
\[
	X \le Y
\]
By Lemma \ref{Y}, we have $Y<0$ over $I$, so
\[
	 X^2 \ge Y^2
\]
Since $\lambda_{k,1}(v,\sigma) \ge \lambda_{k,1}(\omega)$, it follows that, over $I$,
\[
	 X' \le Y'
\]
By Lemma \ref{Y}, we have $Y'<0$ over $I$, which yields $X' < 0$ over $I$.

Note that for the case $V_\omega^{-1}(\mu_{m-1}) \ge P_m$, this completes the proof.
We finish the proof of the lemma by assuming that
\[
	V_\omega^{-1}(\mu_{m-1}) < P_m
\]
and showing that $X'<0$ holds over the interval $[V_\omega^{-1}(\mu_{m-1}),P_m)$.
Note that in this case, $V_\omega^{-1}(\mu_{m-1})$ is in $I$, so
\[
	X'(V_\omega^{-1}(\mu_{m-1})) < 0
\]
Suppose $X'<0$ does not hold over $[V_\omega^{-1}(\mu_{m-1}),P_m)$.
Let $\xi$ be the smallest critical point of $X$ in $[V_\omega^{-1}(\mu_{m-1}),P_m)$.
Then
\[
	[X(\xi)]^2 = k^2 - \lambda_{k,1}(v,\sigma) [v(\xi)]^2
\]
Moreover, $X$ is decreasing over $[V_\omega^{-1}(\mu_{m-1}),\xi]$, so
\[
	X(\xi) < X(V_\omega^{-1}(\mu_{m-1})) \le Y(V_\omega^{-1}(\mu_{m-1})) = y_{m-1} \le 0
\]
Therefore,
\[
	\lambda_{k,1}(\omega)[v(\xi)]^2 + y_{m-1}^2 < \lambda_{k,1}(\omega) [v(\xi)]^2 +[X(\xi)]^2 = k^2
\]
This implies that
\[
	v(\xi) < \sqrt{\frac{k^2-y_{m-1}^2}{\lambda_{k,1}(\omega)}}
\]
However, the inequality $\xi < P_m$, implies
\[
	v(\xi) > \mu_m \ge \sqrt{\frac{k^2-y_{m-1}^2}{\lambda_{k,1}(\omega)}}
\]
This contradiction shows that $X'<0$ over $[V_\omega^{-1}(\mu_{m-1}),P_m)$.
\end{proof}

Proceeding with the proof of Lemma \ref{mm1}, let $m$ be a positive integer.
Let $L_{m-1} \ge r_1-r_2$ be finite and let
\[
	\chi_{m-1}:[0,L_{m-1}] \to \overline{\R^2_+}
\]
be in $\Gamma_{\mu_{m-1}}$.
Write $\chi_{m-1}=(F_{\chi_{m-1}},G_{\chi_{m-1}})$, and define $V_{\chi_{m-1}}=V \circ F_{\chi_{m-1}}$.
Then recall that
\[
	V(r_2) \le \mu_m < V(r_1)
\]
and define
\[
	P_m = \min \bigg\{ t \in [0,L_{m-1}] : V_{\chi_{m-1}}(t)= \mu_m \bigg\}
\]
Define a piecewise smooth function
\[
	\beta_m:[0,L_{m-1}] \to \overline{\R^2_+}
\]
such that
\[
	\beta_m(0)=(r_1,0)
\]
and
\[
	\beta_m'(t) =
	\begin{cases}
		(F_{\chi_{m-1}}'(t), 0) & t \in [0,P_m) \\
		(F_{\chi_{m-1}}'(t), G_{\chi_{m-1}}'(t) & t \in (P_m,L_{m-1}] \\
	\end{cases}
\]

\begin{Lemma}
\label{mab}
Assume $\chi_{m-1}$ is not equal to $\beta_m$ and $\lambda_{k,1}(\chi_{m-1}) \ge \lambda_{k,1}(\omega)$.
Then
\[
	\lambda_{k,1}(\chi_{m-1}) < \lambda_{k,1}(\beta_m)
\]
\end{Lemma}

\begin{proof}
For each $s$ in $[0,1]$, define a piecewise smooth curve
\[
	\zeta_s:[0,L_{m-1}] \to \overline{\R^2_+}
\]
such that
\[
	\zeta_s(0)=(r_1,0)
\]
and
\[
	\zeta_s'(t) =
	\begin{cases}
		(F_{\chi_{m-1}}'(t), s G_{\chi_{m-1}}'(t)) & t \in [0,P_m) \\
		(F_{\chi_{m-1}}'(t), G_{\chi_{m-1}}'(t) & t \in (P_m,L_{m-1}] \\
	\end{cases}
\]
Note that $\zeta_0=\beta_m$ and $\zeta_1=\chi_{m-1}$.

By a theorem of Kong and Zettl \cite[Theorem 3.1]{KZ}, the function
\[
	s \mapsto \lambda_{k,1}(\zeta_s)
\]
is continuous over $(0,1]$.
We begin by observing that this function is also upper semi-continuous at zero.
Let $\epsilon >0$.
There is a function $w$ in $\Lip_0(0,L_{m-1})$ such that
\[
	\frac{ \int_0^{L_{m-1}} \frac{|w'|^2 V_{\chi_{m-1}}}{|\zeta_0 '|} + \frac{k^2 w^2 | \zeta_0' |}{V_{\chi_{m-1}}} \,dt}{ \int_0^{L_{m-1}} w^2 V_{\chi_{m-1}} | \zeta_0' | \,dt} < \lambda_{k,1}(\zeta_0) + \epsilon
\]
By Lebesgue's convergence theorems
\[
	\lim_{s \to 0}
	\frac{ \int_0^{L_{m-1}} \frac{|w'|^2 V_{\chi_{m-1}}}{|\zeta_s '|} + \frac{k^2 w^2 | \zeta_s' |}{V_{\chi_{m-1}}} \,dt}{ \int_0^{L_{m-1}} w^2 V_{\chi_{m-1}} | \zeta_s' | \,dt}
 =
	\frac{ \int_0^{L_{m-1}} \frac{|w'|^2 V_{\chi_{m-1}}}{|\zeta_0 '|} + \frac{k^2 w^2 | \zeta_0' |}{V_{\chi_{m-1}}} \,dt}{ \int_0^{L_{m-1}} w^2 V_{\chi_{m-1}} | \zeta_0' | \,dt}
\]
Therefore,
\[
	\limsup_{s \to 0} \lambda_{k,1}(\zeta_s) < \lambda_{k,1}(\zeta_0) + \epsilon
\]
Since $\epsilon > 0$ is arbitrary, this proves that the function
\[
	s \mapsto \lambda_{k,1}(\zeta_s)
\]
is upper semi-continuous at zero.

We next show that this function is strictly monotonically decreasing.
Fix a point $s_0$ in $(0,1)$.
Let $\phi=\phi_{k,1}$ be the eigenfunction given by Lemma \ref{efex} corresponding to $\zeta_{s_0}$.
We may assume that
\[
	\int_0^{L_{m-1}} | \phi |^2 V_{\chi_{m-1}} | \zeta_{s_0}'| \,dt = 1
\]
Define a function $\Psi:(0,1] \to \R$ by
\[
	\Psi(s) = \frac{ \int_0^{L_{m-1}} \frac{|\phi'|^2 V_{\chi_{m-1}}}{|\zeta_s '|} + \frac{k^2 |\phi|^2 | \zeta_s' |}{V_{\chi_{m-1}}} \,dt}{ \int_0^{L_{m-1}} |\phi|^2 V_{\chi_{m-1}} | \zeta_s' | \,dt}
\]
Then $\Psi(s) \ge \lambda_{k,1}(\zeta_s)$ for all $s$ in $(0,1]$, and
\[
	\Psi(s_0)= \lambda_{k,1}(\zeta_{s_0})
\]
Also $\Psi$ is differentiable at $s_0$ and $\Psi'(s_0)$ is equal to
\[
	\int_0^{P_m} \bigg( -\frac{|\phi'|^2 V_{\chi_{m-1}}}{|\zeta_{s_0} '|^2} + \frac{k^2 |\phi|^2}{V_{\chi_{m-1}}} - \lambda_{k,1}(\zeta_{s_0}) | \phi|^2 V_{\chi_{m-1}} \bigg) \frac{s_0 | G_{\chi_{m-1}}' |^2}{ |\zeta'_{s_0}|} \,dt
\]
In particular, by Lemma \ref{efes} and by the assumption that $\chi_{m-1}$ is not equal to $\beta_m$,
\[
	\Psi'(s_0)<0
\]
Therefore there is a number $s_1$ in $(s_0,1]$ such that, for $s$ in $(s_0,s_1]$,
\[
	\lambda_{k,1}(\zeta_s) \le \Psi(s) < \Psi(s_0) = \lambda_{k,1}(\zeta_{s_0})
\]
It now follows that the function
\[
	s \mapsto \lambda_{k,1}(\zeta_s)
\]
is strictly monotonically decreasing over $[0,1]$.
In particular,
\[
	\lambda_{k,1}(\chi_{m-1}) = \lambda_{k,1}(\zeta_1) < \lambda_{k,1}(\zeta_0) = \lambda_{k,1}(\beta_m)
\]
\end{proof}

Let $L_m^*$ be the length of $\beta_m$.
Define $\ell_m:[0,L_{m-1}] \to [0,L_m^*]$ by
\[
	\ell_m(t) = \int_0^t |\beta_m'(u)| \,du
\]
Define $\rho_m:[0,L_m^*] \to [0,L_{m-1}]$ by
\[
	\rho_m(s) = \min \Big\{ t \in [0,L_{m-1}] : \ell_m(t) = s \Big\}
\]
The function $\rho_m$ may be discontinuous, but the curve $\gamma_m=\beta_m \circ \rho_m$ is Lipschitz continuous, and for all $t$ in $[0,L_{m-1}]$,
\[
	\gamma_m(\ell_m(t)) = \beta_m(t)
\]
Moreover, for almost all $t$ in $[0,L_m^*]$,
\[
	|\gamma_m'(t)|=1
\]

\begin{Lemma}
\label{mbc}
This reparametrization satisfies
\[
	\lambda_{k,1}(\beta_m) \le \lambda_{k,1}(\gamma_m)
\]
\end{Lemma}

\begin{proof}
Write $\beta_m=(F_{\beta_m}, G_{\beta_m})$ and $\gamma_m=(F_{\gamma_m}, G_{\gamma_m})$.
Let $V_{\beta_m}=V \circ F_{\beta_m}$ and $V_{\gamma_m} = V \circ F_{\gamma_m}$.
Let $w$ be a function in $\Lip(0,L_m^*)$ such that
\[
	\frac{\int_0^{L_m^*} \frac{|w'|^2 V_{\gamma_m}}{|\gamma_m'|} + \frac{k^2 |w|^2 | \gamma_m'|}{V_{\gamma_m}} \,dt}{\int_0^{L_m^*} |w|^2 V_{\gamma_m} |\gamma_m'| \,dt} < \infty
\]
Define $v=w \circ \ell_m$.
Then $v$ is in $\Lip_0(0,L_{m-1})$, and changing variables yields
\[
	\frac{\int_0^{L_{m-1}} \frac{|v'|^2 V_{\beta_m}}{|\beta_m'|} + \frac{k^2 |v|^2 | \beta_m'|}{V_{\beta_m}} \,dt}{\int_0^{L_{m-1}} |v|^2 V_{\beta_m} |\beta_m'| \,dt}
	= \frac{\int_0^{L_m^*} \frac{|w'|^2 V_{\gamma_m}}{|\gamma_m'|} + \frac{k^2 |w|^2 | \gamma_m'|}{V_{\gamma_m}} \,dt}{\int_0^{L_m^*} |w|^2 V_{\gamma_m} |\gamma_m'| \,dt}
\]
It follows that
\[
	\lambda_{k,1}(\beta_m) \le \lambda_{k,1}(\gamma_m)
\]
	
\end{proof}

Define
\[
	B_m = \min \bigg\{ t \in [0,L_m^*] : V_{\gamma_m}(t)=\mu_m \bigg\}
\]
Let
\[
	L_m=L_m^*-B_m+V_\omega^{-1}(\mu_m)
\]
Define a piecewise smooth curve
\[
	\chi_m:[0,L_m] \to \R^2_+
\]
by
\[
	\chi_m(t) =
	\begin{cases}
		\omega(t) & t \in [0,V_\omega^{-1}(\mu_m)] \\
		\gamma_m(t+B_m-V_\omega^{-1}(\mu_m)) & t \in [V_\omega^{-1}(\mu_m),L_m] \\
	\end{cases}
\]

\begin{Lemma}
\label{mcm}
Assume $\lambda_{k,1}(\gamma_m) \ge \lambda_{k,1}(\omega)$.
Then
\[
	\lambda_{k,1}(\gamma_m) \le \lambda_{k,1}(\chi_m)
\]
\end{Lemma}

\begin{proof}
Write $\gamma_m=(F_{\gamma_m}, G_{\gamma_m})$ and let $V_{\gamma_m}=V \circ F_{\gamma_m}$.
Define
\[
	I = \bigg\{ t \in [0,B_m] : V_{\gamma_m}(t) > \min_{u \in [0,t]} V_{\gamma_m}(u) \bigg\}
\]
There are disjoint open intervals $I_1, I_2, I_3, \ldots$ such that
\[
	I = \bigcup_j I_j
\]
Let
\[
	J=[0,L_m^*] \setminus I
\]
Define a Lipschitz continuous function
\[
	h:[0,L_m^*] \to [0,L_m]
\]
such that $h(0)=0$ and for almost all $t$ in $[0,L_m^*]$,
\[
	h'(t) =
	\begin{cases}
		1 & t \in  J \\
		0 & t \in I \\
	\end{cases}
\]
Since the length of $I$ is $B_m-V_\omega^{-1}(\mu_m)$, this map is surjective.
Define a curve
\[
	\eta = \chi_m \circ h
\]
For any function $w$ in $\Lip_0(0,L_m)$, the function $v=w \circ h$ is in $\Lip_0(0,L_m^*)$ and satisfies
\[
	\frac{\int_0^{L_m^*} \frac{|v'|^2 V_{\eta}}{| \eta'|}+\frac{k^2|v|^2 | \eta'|}{V_{\eta}} \,dt}{\int_0^{L_m^*} |v|^2 V_{\eta} | \eta'| \,dt}
	=
	\frac{\int_0^{L_m} \frac{|w'|^2 V_{\chi_m}}{| \chi_m'|}+\frac{k^2|w|^2 | \chi_m'|}{V_{\chi_m}} \,dt}{\int_0^{L_m} |w|^2 V_{\chi_m} | \chi_m'| \,dt}
\]
Furthermore, if $v$ is a function in $\Lip_0(0,L_m^*)$ such that
\[
	\frac{\int_0^{L_m^*} \frac{|v'|^2 V_{\eta}}{| \eta'|}+\frac{k^2|v|^2 | \eta'|}{V_{\eta}} \,dt}{\int_0^{L_m^*} |v|^2 V_{\eta} | \eta'| \,dt} < \infty
\]
then $v$ is constant over each interval $I_1,I_2,I_3,\ldots$, and it follows that there is a function $w$ in $\Lip_0(0,L_m)$ such that $v=w \circ h$.
Therefore
\[
	\lambda_{k,1}(\eta)=\lambda_{k,1}(\chi_m)
\]	

Note that $V_{\gamma_m}$ and $V_\eta$ agree over $J$.
This implies that
\[
	\lambda_{k,1}(\eta) = \lambda_{k,1}(V_{\gamma_m}, |\eta'|)
\]
For a point $t$ in $I$, let $d(t,J)$ be the distance from $t$ to $J$.
For each $s$ in $(0,1]$, define a continuous function
\[
	\sigma_s:[0,L_m^*] \to \R
\]
by
\[
	\sigma_s(t) = \max \bigg\{ 1-\frac{d(t,J)}{s} , s \bigg\}
\]
Define
\[
	\sigma_0:[0,L_m^*] \to \R
\]
by
\[
	\sigma_0(t)=
	\begin{cases}
		1 & t \in J \\
		0 & t \in I \\
	\end{cases}
\]
Note that $\sigma_0 = | \eta' |$ almost everywhere in $[0,L_m^*]$.
In particular,
\[
	\lambda_{k,1}(\eta) = \lambda_{k,1}(V_{\gamma_m}, \sigma_0 )
\]
Also,
\[
	\lambda_{k,1}(\gamma_m) = \lambda_{k,1}(\gamma_m, \sigma_1)
\]
By a theorem of Kong and Zettl \cite[Theorem 3.1]{KZ}, the function
\[
	s \mapsto \lambda_{k,1}(V_{\gamma_m},\sigma_s)
\]
is continuous over $(0,1]$.
We now observe that it is upper semi-continuous at zero.
Let $\epsilon>0$.
There is a function $w$ in $\Lip_0(0,L_m^*)$ such that
\[
	\frac{\int_0^{L_m^*} \frac{|w'|^2 V_{\gamma_m}}{\sigma_0}+\frac{k^2|w|^2 \sigma_0}{V_{\gamma_m}} \,dt}{\int_0^{L_m^*} |w|^2 V_{\gamma_m} \sigma_0 \,dt} < \lambda_{k,1}(V_{\gamma_m},\sigma_0) + \epsilon
\]
By Lebesgue's convergence theorems,
\[
	\lim_{s \to 0} \frac{\int_0^{L_m^*} \frac{|w'|^2 V_{\gamma_m}}{\sigma_s}+\frac{k^2|w|^2 \sigma_s}{V_{\gamma_m}} \,dt}{\int_0^{L_m^*} |w|^2 V_{\gamma_m} \sigma_s \,dt}
	=
	\frac{\int_0^{L_m^*} \frac{|w'|^2 V_{\gamma_m}}{\sigma_0}+\frac{k^2|w|^2 \sigma_0}{V_{\gamma_m}} \,dt}{\int_0^{L_m^*} |w|^2 V_{\gamma_m} \sigma_0 \,dt}
\]
Therefore,
\[
	\limsup_{s \to 0} \lambda_{k,1}(V_{\gamma_m}, \sigma_s) < \lambda_{k,1}(V_{\gamma_m},\sigma_0)+ \epsilon
\]
Since $\epsilon>0$ is arbitrary, this shows that the function
\[
	s \mapsto \lambda_{k,1}(V_{\gamma_m}, \sigma_s)
\]
is upper semi-continuous at zero.

We next show that this function is monotonically decreasing over $[0,1]$.
Fix $s_0$ in $(0,1]$.
Let $\phi=\phi_{k,1}$ be the eigenfunction corresponding to $V_{\gamma_m}$ and $\sigma_{s_0}$ given by Lemma \ref{efex}.
We may assume that
\[
	\int_0^{L_m^*} |\phi|^2 V_{\gamma_m} \sigma_{s_0} \,dt = 1
\]
Define a function
\[
	\Psi:[0,1] \to \R
\]
by
\[
	\Psi(s) = \frac{ \int_0^{L_m^*} \frac{|\phi'|^2 V_{\gamma_m}}{\sigma_s} + \frac{k^2 |\phi|^2 \sigma_s}{V_{\gamma_m}} \,dt}{ \int_0^{L_m^*} |\phi|^2 V_{\gamma_m} \sigma_s \,dt}
\]
Then for all $s$ in $[0,1]$,
\[
	\Psi(s) \ge \lambda_{k,1}(V_{\gamma_m},\sigma_s)
\]
Also
\[
	\Psi(s_0) = \lambda_{k,1}(V_{\gamma_m},\sigma_{s_0})
\]
Define a function
\[
	\dot \sigma_{s_0} : [0,L_m^*] \to \R
\]
by
\[
	\dot \sigma_{s_0}(t) =
	\begin{cases}
		0 & \text{if } d(t,J) = 0 \\
		1 & \text{if } d(t,J) > s_0 - s_0^2 \\
		\frac{d(t,J)}{s_0^2} & \text{if } 0 < d(t,J) \le s_0-s_0^2 \\
	\end{cases}
\]
Then $\Psi$ is differentiable at $s_0$ and $\Psi'(s_0)$ is equal to
\[
	\int_0^{L_m^*} \bigg( -\frac{|\phi'|^2 V_{\gamma_m}}{\sigma_{s_0}^2} + \frac{k^2 |\phi|^2}{V_{\gamma_m}} - \lambda_{k,1}(V_{\gamma_m},\sigma_{s_0}) V_{\gamma_m} | \phi|^2 \bigg) \dot \sigma_{s_0} \,dt
\]
In particular, by Lemma \ref{efes},
\[
	\Psi'(s_0) \le 0
\]
This proves that the function
\[
	s \mapsto \lambda_{k,1}(V_{\gamma_m}, \sigma_s)
\]
has non-positive upper right Dini derivative at every point $s_0$ in $(0,1]$. 
We have seen this function is continuous over $(0,1]$ and upper semi-continuous at zero, so this implies that it is monotonically decreasing over $[0,1]$.
Now
\[
	\lambda_{k,1}(\gamma_m) = \lambda_{k,1}(V_{\gamma_m}, \sigma_1) \le \lambda_{k,1}(V_{\gamma_m}, \sigma_0) = \lambda_{k,1}(\eta) = \lambda_{k,1}(\chi_m)
\]

\end{proof}

We can now prove Lemma \ref{mm1}, concluding the argument.

\begin{proof}[Proof of Lemma 4.4]
Note that $\chi_m$ is in $\Gamma_{\mu_m}$.
Suppose $\chi_m$ is not equal to $\chi_{m-1}$ and
\[
	\lambda_{k,1}(\chi_{m-1}) \ge \lambda_{k,1}(\omega)
\]
Then $\chi_{m-1}$ is not equal to $\beta_m$, so by Lemmas \ref{mab}, \ref{mbc}, and \ref{mcm},
\[
	\lambda_{k,1}(\chi_{m-1}) < \lambda_{k,1}(\beta_m) \le \lambda_{k,1}(\gamma_m)  \le \lambda_{k,1}(\chi_m)
\]
\end{proof}

\end{document}